\documentclass[11pt]{amsart}

\usepackage{amsmath,amssymb,amsthm,mathtools}
\usepackage{mathrsfs}
\usepackage{microtype}
\usepackage[hidelinks]{hyperref}
\hypersetup{
  pdftitle={Bismut-Torsion-Parallel Hermitian Manifolds with Constant Chern Holomorphic Sectional Curvature},
  pdfauthor={Haohao Wang },
  pdfkeywords={Hermitian manifold, Chern holomorphic sectional curvature, Bismut connection, parallel torsion, balanced metric, complex Lie algebra}
}
\usepackage{enumitem}

\allowdisplaybreaks
\numberwithin{equation}{section}

\newtheorem{theorem}{Theorem}[section]
\newtheorem{proposition}[theorem]{Proposition}
\newtheorem{lemma}[theorem]{Lemma}
\newtheorem{corollary}[theorem]{Corollary}
\newtheorem{conjecture}[theorem]{Conjecture}
\theoremstyle{definition}

\theoremstyle{remark}
\newtheorem{remark}[theorem]{Remark}

\newcommand{\C}{\mathbb C}
\newcommand{\R}{\mathbb R}
\newcommand{\Id}{\operatorname{Id}}
\newcommand{\End}{\operatorname{End}}
\newcommand{\Der}{\operatorname{Der}}
\newcommand{\Rad}{\operatorname{Rad}}
\newcommand{\ad}{\operatorname{ad}}
\newcommand{\tr}{\operatorname{tr}}
\newcommand{\pr}{\operatorname{pr}}
\newcommand{\Span}{\operatorname{span}}

\newcommand{\Lie}{\mathfrak g}
\newcommand{\rad}{\mathfrak r}
\newcommand{\der}{\mathfrak d}
\newcommand{\ip}[2]{\left\langle #1,#2\right\rangle}
\newcommand{\norm}[1]{\left\lVert #1\right\rVert}
\newcommand{\abs}[1]{\left\lvert #1\right\rvert}

\title[BALANCED BTP METRICS AND CONSTANT CHERN HSC]
{ BISMUT-TORSION-PARALLEL HERMITIAN MANIFOLDS WITH CONSTANT CHERN HOLOMORPHIC SECTIONAL CURVATURE}

\author{Haohao Wang}
\date{}

\begin{document}

\begin{abstract}
A well-known conjecture in complex geometry states that a compact Hermitian
manifold with constant Chern holomorphic sectional curvature must be
K\"ahler when the constant is nonzero and Chern flat when the constant is
zero.  The conjecture is known in complex dimension two and in several
special classes in higher dimensions.  For Hermitian metrics with
Bismut-parallel torsion, the non-balanced case and the balanced threefold
case were established by Chen--Zheng, while the balanced fourfold case was
settled recently by Wang--Zheng.  In this article, we prove the nonzero
case for balanced Bismut-torsion-parallel Hermitian manifolds in every
complex dimension. As a corollary, we prove that every BTP Hermitian manifold with nonzero constant Chern holomorphic sectional curvature is Kähler.
\end{abstract}

\maketitle

\section{Introduction}

Let $(M^n,J,g)$ be a Hermitian manifold of complex dimension $n\geq 2$.
We denote by $\nabla^c$ its Chern connection, by $R^c$ its Chern curvature,
and by $T=T^c$ its Chern torsion.  For a nonzero tangent vector
$X\in T^{1,0}M$, the Chern holomorphic sectional curvature is defined by
\begin{equation}\label{eq:HSC-definition}
 H^c(X)=\frac{R^c(X,\overline X,X,\overline X)}{\norm{X}^4}.
\end{equation}
When $g$ is K\"ahler, the curvature tensor has the usual K\"ahler
symmetries, so the values of $H^c$ determine the entire curvature tensor.
The complete K\"ahler manifolds with constant holomorphic sectional
curvature are precisely the complex space forms; their universal covers
are complex projective space, complex Euclidean space, or complex
hyperbolic space, equipped with a suitable scaling of the standard
metric.  For a non-K\"ahler Hermitian metric, however, $R^c$ generally
does not satisfy all K\"ahler symmetries, and $H^c$ determines only an
appropriate symmetrization of $R^c$.

A long-standing conjecture in non-K\"ahler geometry is the following.

\begin{conjecture}
\label{conj:constant-HSC}
Let $(M^n,J,g)$ be a compact Hermitian manifold with $n\geq 2$.  Suppose
that $H^c\equiv c$ is constant.  If $c\neq 0$, then $g$ is K\"ahler; if
$c=0$, then $g$ is Chern flat.
\end{conjecture}

The conclusion in the zero-curvature case cannot in general be strengthened
to K\"ahlerness: compact non-K\"ahler Chern-flat manifolds exist in
complex dimension at least three \cite{Boothby}.

The conjecture was established in complex dimension two by Balas and
Gauduchon for $c\leq 0$ and by Apostolov--Davidov--Mu\v{s}karov in the
remaining case \cite{Balas,BalasGauduchon,ADM}.  In higher
dimensions it remains open in general, although
several important special cases are known.  Tang treated Chern
K\"ahler-like metrics; Chen--Chen--Nie proved the nonpositive case for
locally conformally K\"ahler metrics; Zhou--Zheng obtained Chern flatness
for compact Hermitian threefolds with vanishing real bisectional
curvature; and Ma--Nie considered normal balanced threefolds with
nonpositive constant Chern holomorphic sectional curvature
\cite{Tang,CCN,ZhouZheng,MaNie}.  More recently, the nonpositive case for
compact balanced threefolds has been proved without the normality assumption
\cite{ChenLi}.

The present paper concerns Hermitian metrics whose Bismut connection has
parallel torsion.  Recall that the Bismut connection $\nabla^b$ is the
unique Hermitian connection whose torsion is totally skew-symmetric \cite{Bismut}.  A
Hermitian metric is called \emph{Bismut torsion parallel}, or \emph{BTP},
if
\[
 \nabla^b T^b=0,
\]
where $T^b$ denotes the Bismut torsion.  Equivalently, one has
$\nabla^bT=0$ for the Chern torsion.  The BTP class contains, among other examples,
Bismut K\"ahler-like metrics, Bismut-flat metrics, Vaisman metrics, and compact
quotients of certain complex Lie groups; see \cite{ChenZhengBTP,ZhaoZhengBTP}
and the references therein.

Chen and Zheng proved Conjecture~\ref{conj:constant-HSC} for all
non-balanced BTP metrics and, using the classification of Zhao--Zheng,
for balanced BTP threefolds \cite{ChenZhengBTP}.  Their work left the
balanced case in complex dimension at least four open.  Wang and Zheng
subsequently settled the balanced fourfold case by a dimension-specific
analysis \cite{WangZhengFourfold}.  Our main result
is the following arbitrary-dimensional statement.

\begin{theorem}\label{thm:main}
Let $(M^n,J,g)$ be a balanced Hermitian manifold of complex dimension
$n\geq 2$.  Assume that
\begin{equation}\label{eq:main-assumptions}
 \nabla^bT=0,
 \qquad
 H^c\equiv c\neq 0.
\end{equation}
Then $T=0$.  Consequently, $g$ is K\"ahler.
\end{theorem}

We recently learned through personal communication that Professors Fangyang Zheng and Qingsong Wang have also independently arrived at the above result using different methods.

No compactness, completeness, homogeneity, special unitary frame, or
classification theorem is used in the proof.  Combining
Theorem~\ref{thm:main} with the non-balanced result of Chen--Zheng\cite{ChenZhengBTP}, where the proof process is pointwise, gives
the following consequence for the original compact conjecture. 

\begin{corollary}\label{cor:all-BTP}
Let $(M^n,J,g)$ be a BTP Hermitian manifold, not necessarily compact
or complete.  If its Chern
holomorphic sectional curvature is a nonzero constant, then $g$ is
K\"ahler.
\end{corollary}

The proof of Theorem~\ref{thm:main} is pointwise.  Fix $p\in M$ and put
$V=T^{1,0}_pM$.  The BTP identities imply that
\[
 [x,y]:=T_p(x,y),\qquad x,y\in V,
\]
defines a complex Lie bracket.  Since $T$ is $\nabla^b$-parallel, every
Bismut curvature operator $R^b(X,\overline Y)$ is a derivation of this
Lie algebra.  Balancedness is precisely the unimodularity condition
$\tr(\ad_x)=0$.  Finally, the constant Chern holomorphic sectional
curvature assumption reconstructs the full Bismut curvature from $c$ and
quadratic expressions in the bracket.

The resulting finite-dimensional argument has three stages.  First, a
contracted curvature derivation and a mixed calculation along the solvable
radical exclude every nonzero semisimple quotient.  Thus the torsion Lie
algebra is solvable.  Second, Lie's theorem and a derivation invariance
property of the adjoint weights exclude the solvable non-nilpotent case.
Thus the algebra is nilpotent.  Third, a nonzero nilpotent Lie algebra has
nonzero center.  Curvature in a central direction first eliminates every
bracket component taking values in that central direction and then forces
the entire bracket to vanish.

The assumption $c\neq 0$ is essential for this argument. When $c=0$, non-K\"ahler Chern-flat BTP examples already exist in
complex dimension three \cite{ZhaoZhengThreefold}.  Although the present argument does not prove the full zero-curvature conjecture, its $c$-independent part still yields a structural description of the pointwise torsion Lie algebra.
In particular, the radical admits an orthogonal semisimple
complement, the center is orthogonal to the derived algebra,
and nilpotent torsion forces flat K\"ahlerness.

The paper is organized as follows.  Section~\ref{sec:preliminaries}
sets up the Hermitian and Lie-algebraic conventions and records the two
external BTP curvature identities.  Sections~\ref{sec:operators} and
\ref{sec:contracted} introduce two quadratic torsion operators and the
contracted curvature derivation.  Sections~\ref{sec:radical},
\ref{sec:weights}, and \ref{sec:center} carry out the three algebraic
stages described above.  The proof of the main theorem and several
consequences are given in Section~\ref{sec:conclusion}.

\section{Preliminaries}\label{sec:preliminaries}

\subsection{Chern torsion, Bismut curvature, and balancedness}

We extend $g$ complex bilinearly to $TM\otimes_{\R}\C$.  Let
$e=(e_1,\ldots,e_n)$ be a local unitary frame of $T^{1,0}M$, with dual
coframe $\varphi=(\varphi^1,\ldots,\varphi^n)$.  Thus
\[
 \ip{e_i}{\overline{e_j}}=\delta_{ij},
 \qquad
 \ip{e_i}{e_j}=0.
\]
We use the Chern torsion normalization
\begin{equation}\label{eq:torsion-normalization}
 T(e_i,e_k)=\sum_j T^j_{ik}e_j,
 \qquad
 T^j_{ik}=-T^j_{ki},
\end{equation}
so that the torsion forms are
\begin{equation}\label{eq:torsion-forms}
 \tau^j=\frac12\sum_{i,k}T^j_{ik}\,\varphi^i\wedge\varphi^k.
\end{equation}
This is the normalization used in the curvature formula of
Chen--Zheng.  It is twice the coefficient convention adopted in some
earlier papers on Hermitian curvature.  All sums of squared torsion
components below are ordered sums; in particular,
\begin{equation}\label{eq:T-norm}
 \abs{T}^2=\sum_{i,k,j}\abs{T^j_{ik}}^2.
\end{equation}

For the Bismut curvature we use the convention below; compare
\cite{YangZheng,ZhaoZhengBTP}.
\begin{equation}\label{eq:Bismut-curvature-components}
 R^b_{i\bar j k\bar\ell}
 =\ip{R^b(e_i,\overline{e_j})e_k}{\overline{e_\ell}},
\end{equation}
where
\[
 R^b(X,Y)=\nabla^b_X\nabla^b_Y-\nabla^b_Y\nabla^b_X-\nabla^b_{[X,Y]}.
\]
We write
\begin{equation}\label{eq:Dij-definition}
 D_{i\bar j}:=R^b(e_i,\overline{e_j})\big|_{T^{1,0}M}.
\end{equation}

Gauduchon's torsion $(1,0)$-form $\eta$ has components
\begin{equation}\label{eq:eta-components}
 \eta_k=\sum_iT^i_{ik}.
\end{equation}
If
\[
 \omega=\sqrt{-1}\sum_i\varphi^i\wedge\overline{\varphi^i}
\]
is the Hermitian form, then
\[
 g\text{ is balanced}
 \quad\Longleftrightarrow\quad
 \eta=0
 \quad\Longleftrightarrow\quad
 d\omega^{n-1}=0.
\]
This equivalence is standard; see, for example, \cite{Gauduchon}.

The BTP condition is usually written $\nabla^bT^b=0$.  Since the Bismut
and Chern torsions determine one another algebraically, this is equivalent to
\begin{equation}\label{eq:BTP-Chern-torsion}
 \nabla^bT=0.
\end{equation}
See \cite{ChenZhengBTP,ZhaoZhengBTP}.  We use the latter form throughout.

\subsection{Polarization of constant Chern holomorphic sectional curvature}

For a tensor $P_{i\bar j k\bar\ell}$, define its holomorphic sectional
symmetrization by
\begin{equation}\label{eq:symmetrization}
 \widehat P_{i\bar j k\bar\ell}
 =\frac14\bigl(
 P_{i\bar j k\bar\ell}
 +P_{k\bar j i\bar\ell}
 +P_{i\bar\ell k\bar j}
 +P_{k\bar\ell i\bar j}
 \bigr).
\end{equation}
The usual complex polarization argument gives the following statement;
compare \cite{Balas}.

\begin{lemma}\label{lem:polarization}
The condition $H^c\equiv c$ is equivalent to
\begin{equation}\label{eq:polarized-HSC}
 \widehat R^c_{i\bar j k\bar\ell}
 =\frac c2\bigl(\delta_{ij}\delta_{k\ell}
 +\delta_{i\ell}\delta_{kj}\bigr).
\end{equation}
Equivalently,
\begin{align}\label{eq:polarized-HSC-expanded}
 &R^c_{i\bar j k\bar\ell}
 +R^c_{k\bar j i\bar\ell}
 +R^c_{i\bar\ell k\bar j}
 +R^c_{k\bar\ell i\bar j}\notag\\
 &\hspace{35mm}=2c\bigl(\delta_{ij}\delta_{k\ell}
 +\delta_{i\ell}\delta_{kj}\bigr).
\end{align}
\end{lemma}

\begin{proof}
For $X=\sum_i x_ie_i$, the identity $H^c(X)=c$ is equivalent to
\[
 \sum_{i,j,k,\ell}R^c_{i\bar j k\bar\ell}
 x_i\overline{x_j}x_k\overline{x_\ell}
 =c\left(\sum_i\abs{x_i}^2\right)^2.
\]
Only the symmetrization in \eqref{eq:symmetrization} contributes to the
left-hand quartic form.  Polarizing this Hermitian quartic identity gives
\eqref{eq:polarized-HSC}.  Conversely, substituting the same vector in
all four entries immediately recovers $H^c\equiv c$.
\end{proof}

\subsection{The pointwise torsion Lie algebra}

The first external input is the following quadratic identity for BTP
metrics \cite[Proposition~1.5]{ZhaoZhengBTP}.

\begin{proposition}\label{prop:BTP-Jacobi}
On a BTP Hermitian manifold, the Chern torsion satisfies
\begin{equation}\label{eq:BTP-Jacobi}
 \sum_r\left(
 T^r_{ij}T^\ell_{rk}
 +T^r_{jk}T^\ell_{ri}
 +T^r_{ki}T^\ell_{rj}
 \right)=0
\end{equation}
for all $i,j,k,\ell$.
\end{proposition}

\begin{remark}
Equation~\eqref{eq:BTP-Jacobi} is homogeneous of degree two in the
torsion coefficients.  Hence a uniform factor change in the torsion
normalization does not affect the identity, although such a change does
affect the numerical coefficients in the curvature reconstruction
formula below.
\end{remark}

Fix $p\in M$ and set $V=T^{1,0}_pM$.  Define
\begin{equation}\label{eq:torsion-bracket}
 [x,y]_p:=T_p(x,y),\qquad x,y\in V.
\end{equation}

\begin{corollary}\label{cor:torsion-Lie-algebra}
The operation \eqref{eq:torsion-bracket} defines a complex Lie algebra
structure on $V$.
\end{corollary}

\begin{proof}
Complex bilinearity and skew-symmetry follow from the type and
skew-symmetry of the Chern torsion.  In the basis $e_1,\ldots,e_n$, the
coefficient of $e_\ell$ in
\[
 [[e_i,e_j],e_k]+[[e_j,e_k],e_i]+[[e_k,e_i],e_j]
\]
is precisely the left-hand side of \eqref{eq:BTP-Jacobi}.
\end{proof}

We call
\[
 \Lie_p=(T^{1,0}_pM,T_p)
\]
the \emph{torsion Lie algebra} at $p$.

\subsection{Bismut curvature operators are derivations}

Recall that a complex-linear map $D:\Lie\to\Lie$ is a derivation if
\begin{equation}\label{eq:derivation-definition}
 D[x,y]=[Dx,y]+[x,Dy]
\end{equation}
for all $x,y\in\Lie$.  The derivations form a Lie algebra denoted by
$\Der(\Lie)$.

\begin{lemma}\label{lem:curvature-derivations}
Let $p\in M$ and $X,Y\in T^{1,0}_pM$.  If $\nabla^bT=0$, then
\begin{equation}\label{eq:curvature-is-derivation}
 R^b(X,\overline Y)\in\Der(\Lie_p).
\end{equation}
\end{lemma}

\begin{proof}
The connection $\nabla^b$ induces a connection on
$\Lambda^2(T^{1,0}M)^*\otimes T^{1,0}M$.  Since $T$ is parallel, the
induced curvature annihilates $T$.  Hence, for $u,v\in T^{1,0}_pM$,
\begin{align*}
 0
 &=(R^b(X,\overline Y)\cdot T)(u,v)\\
 &=R^b(X,\overline Y)T(u,v)
   -T(R^b(X,\overline Y)u,v)
   -T(u,R^b(X,\overline Y)v).
\end{align*}
This is exactly \eqref{eq:derivation-definition} for the torsion bracket.
\end{proof}

\subsection{Balancedness and unimodularity}

For a complex Lie algebra $\Lie$, write
\[
 \ad_x(y)=[x,y].
\]
The algebra is called \emph{unimodular} if
\[
 \tr(\ad_x)=0\qquad\text{for every }x\in\Lie.
\]

\begin{lemma}\label{lem:balanced-unimodular}
A Hermitian metric is balanced if and only if every pointwise torsion Lie
algebra $\Lie_p$ is unimodular.
\end{lemma}

\begin{proof}
In a unitary frame,
\[
 \ad_{e_k}(e_i)=[e_k,e_i]=\sum_jT^j_{ki}e_j.
\]
Therefore
\[
 \tr(\ad_{e_k})=\sum_iT^i_{ki}
 =-\sum_iT^i_{ik}=-\eta_k.
\]
The conclusion follows from \eqref{eq:eta-components}.
\end{proof}

\subsection{Curvature reconstruction}

The second external geometric input reconstructs the full Bismut curvature
from the constant $c$ and the Chern torsion; it is the $t=0$ case of
\cite[Lemma~8]{ChenZhengBTP}.

\begin{proposition}[Chen--Zheng]\label{prop:reconstruction}
Assume that $(M^n,J,g)$ is BTP and that $H^c\equiv c$.  Then, under every
local unitary frame,
\begin{align}
 R^b_{i\bar j k\bar\ell}
={}&\frac c2\bigl(\delta_{ij}\delta_{k\ell}
       +\delta_{i\ell}\delta_{kj}\bigr)
 -\frac12\sum_rT^r_{ik}\overline{T^r_{j\ell}}
 \notag\\
&-\frac34\sum_r\left(
 T^j_{ir}\overline{T^k_{\ell r}}
 +T^\ell_{kr}\overline{T^i_{jr}}
 \right)
 \notag\\
&+\frac14\sum_r\left(
 T^\ell_{ir}\overline{T^k_{jr}}
 +T^j_{kr}\overline{T^i_{\ell r}}
 \right).
\label{eq:reconstruction}
\end{align}
\end{proposition}

\begin{proof}
We recall the algebraic derivation in order to fix the signs and
normalization.  Put
\begin{align*}
 w_{i\bar j k\bar\ell}
 &=\sum_rT^r_{ik}\overline{T^r_{j\ell}},\\
 (v_i^j)_{i\bar j k\bar\ell}
 &=\sum_rT^j_{ir}\overline{T^k_{\ell r}},\\
 (v_k^\ell)_{i\bar j k\bar\ell}
 &=\sum_rT^\ell_{kr}\overline{T^i_{jr}},\\
 (v_i^\ell)_{i\bar j k\bar\ell}
 &=\sum_rT^\ell_{ir}\overline{T^k_{jr}},\\
 (v_k^j)_{i\bar j k\bar\ell}
 &=\sum_rT^j_{kr}\overline{T^i_{\ell r}}.
\end{align*}
For BTP metrics, the Bismut curvature has the pair symmetry
\[
 R^b_{i\bar j k\bar\ell}=R^b_{k\bar\ell i\bar j}
\]
and satisfies
\begin{align}\label{eq:BTP-curvature-skew}
 R^b_{i\bar j k\bar\ell}-R^b_{k\bar j i\bar\ell}
 =-w-v_i^j-v_k^\ell+v_i^\ell+v_k^j.
\end{align}
Consequently, the symmetrized Bismut curvature is
\begin{equation}\label{eq:symmetrized-Bismut}
 \widehat R^b
 =R^b+\frac12\bigl(w+v_i^j+v_k^\ell-v_i^\ell-v_k^j\bigr).
\end{equation}
The symmetrized Chern--Bismut comparison formula is
\begin{equation}\label{eq:symmetrized-comparison}
 \widehat R^c-\widehat R^b
 =\frac14\bigl(v_i^j+v_k^\ell+v_i^\ell+v_k^j\bigr).
\end{equation}
Substitute \eqref{eq:polarized-HSC} into
\eqref{eq:symmetrized-comparison} and then use
\eqref{eq:symmetrized-Bismut}.  Solving for $R^b$ gives
\[
 R^b=\widehat R^c-\frac12w
 -\frac34(v_i^j+v_k^\ell)
 +\frac14(v_i^\ell+v_k^j),
\]
which is exactly \eqref{eq:reconstruction}.
\end{proof}

\begin{remark}\label{rem:pointwise-after-reconstruction}
Once \eqref{eq:reconstruction} is available, all arguments below take
place at a single point.  No derivatives of $c$, integration, or global
topology occur.
\end{remark}

\subsection{Lie-algebraic notation and standard facts}

For a finite-dimensional complex Lie algebra $\Lie$, its center and
derived algebra are
\[
 Z(\Lie)=\{z\in\Lie:[z,x]=0\ \text{for all }x\in\Lie\},
 \qquad
 [\Lie,\Lie]=\Span\{[x,y]:x,y\in\Lie\}.
\]
The derived series and lower central series are
\[
 \Lie^{(0)}=\Lie,
 \qquad
 \Lie^{(m+1)}=[\Lie^{(m)},\Lie^{(m)}],
\]
and
\[
 \gamma_1(\Lie)=\Lie,
 \qquad
 \gamma_{m+1}(\Lie)=[\Lie,\gamma_m(\Lie)],
\]
respectively.  The algebra is \emph{solvable} if $\Lie^{(m)}=0$ for some
$m$ and \emph{nilpotent} if $\gamma_m(\Lie)=0$ for some $m$.  Every
nilpotent Lie algebra is solvable, but the converse is false.

The \emph{solvable radical} $\Rad(\Lie)$ is the largest solvable ideal.
The quotient $\Lie/\Rad(\Lie)$ is semisimple.  We record the elementary
facts needed later.

\begin{lemma}\label{lem:basic-Lie-facts}
Let $\Lie$ be a finite-dimensional complex Lie algebra.
\begin{enumerate}[label=\textup{(\roman*)}]
\item If $D\in\Der(\Lie)$, then $\exp(tD)\in\operatorname{Aut}(\Lie)$ for
      every $t\in\R$.
\item The radical $\Rad(\Lie)$ is preserved by every derivation.
\item Every derivation of a complex semisimple Lie algebra is inner and
      has trace zero.
\item If $\Lie\neq 0$ is nilpotent, then $Z(\Lie)\neq 0$.
\end{enumerate}
\end{lemma}

\begin{proof}
For (i), let $F_t=\exp(tD)$.  Differentiating
$F_t^{-1}[F_tx,F_ty]$ and using the derivation identity shows that it is
constant in $t$ and equal to $[x,y]$ at $t=0$.  Thus
$F_t[x,y]=[F_tx,F_ty]$.

For (ii), every automorphism sends a solvable ideal to a solvable ideal.
By maximality of the radical,
$F_t(\Rad(\Lie))=\Rad(\Lie)$.  Differentiating at $t=0$ gives
$D(\Rad(\Lie))\subseteq\Rad(\Lie)$.

Statement (iii) is standard: if $\Lie$ is complex semisimple, then
$\Der(\Lie)=\ad(\Lie)$.  Since $\Lie=[\Lie,\Lie]$, each inner derivation
is a sum of commutators of endomorphisms, and hence is traceless.

For (iv), choose the last nonzero term of the lower central series:
$\gamma_m(\Lie)\neq 0$ and $\gamma_{m+1}(\Lie)=0$.  Then
$[\Lie,\gamma_m(\Lie)]=0$, so
$0\neq\gamma_m(\Lie)\subseteq Z(\Lie)$.
\end{proof}

We shall also use Lie's theorem and Engel's theorem \cite{Humphreys}.  Lie's theorem says
that every finite-dimensional complex representation of a solvable Lie
algebra can be simultaneously upper triangularized.  Engel's theorem
says that if $\ad_x$ is nilpotent for every $x\in\Lie$, then $\Lie$ is
nilpotent.  We refer to Humphreys for these standard results.

\section{The pointwise torsion algebra and two quadratic operators}
\label{sec:operators}

Fix a point $p\in M$ and write
\[
 V=T^{1,0}_pM,
 \qquad
 \Lie=(V,[\ ,\ ]),
 \qquad
 [x,y]=T_p(x,y).
\]
All arguments through Section~\ref{sec:center} take place in the
Hermitian vector space $V$.  Define complex-linear endomorphisms
$A,B\in\End_{\C}(V)$ by
\begin{equation}\label{eq:A-B-components}
 A_{i\bar j}=\sum_{r,s}T^r_{is}\overline{T^r_{js}},
 \qquad
 B_{i\bar j}=\sum_{r,s}T^j_{rs}\overline{T^i_{rs}}.
\end{equation}

\begin{lemma}\label{lem:A-B-properties}
The operators $A$ and $B$ are nonnegative Hermitian endomorphisms and
\begin{equation}\label{eq:A-B-kernels}
 \ker A=Z(\Lie),
 \qquad
 \ker B=[\Lie,\Lie]^\perp.
\end{equation}
Moreover,
\begin{equation}\label{eq:A-B-traces}
 \tr A=\tr B=\sum_{i,j,r}\abs{T^r_{ij}}^2.
\end{equation}
\end{lemma}

\begin{proof}
For $x=\sum_ix_ie_i$,
\begin{align*}
 \ip{Ax}{\overline x}
 &=\sum_s\norm{T(x,e_s)}^2\geq 0,\\
 \ip{Bx}{\overline x}
 &=\sum_{r,s}\abs{\ip{T(e_r,e_s)}{\overline x}}^2\geq 0.
\end{align*}
Thus $Ax=0$ if and only if $T(x,e_s)=0$ for every $s$, which is precisely
$x\in Z(\Lie)$.  Likewise, $Bx=0$ if and only if $x$ is orthogonal to
every bracket, that is, $x\in[\Lie,\Lie]^\perp$.  Finally, both traces
are the same complete ordered sum of squared torsion components after a
relabeling of the three indices.
\end{proof}

\begin{remark}\label{rem:input-output}
The two kernels have different meanings.  The condition $x\in\ker A$
eliminates $x$ as a nontrivial \emph{input} of the bracket, whereas
$x\in\ker B$ eliminates the $x$-component of every bracket
\emph{output}.  Neither condition implies the other in general.  For
example, a central vector may still occur as the value of a bracket, as
in the Heisenberg algebra.
\end{remark}

\section{The contracted curvature derivation}\label{sec:contracted}

Define
\begin{equation}\label{eq:S-definition}
 S:=\sum_{i=1}^nD_{i\bar i}
 =\sum_{i=1}^nR^b(e_i,\overline{e_i})\big|_V.
\end{equation}
By Lemma~\ref{lem:curvature-derivations},
\begin{equation}\label{eq:S-derivation}
 S\in\Der(\Lie).
\end{equation}

\begin{lemma}\label{lem:S-formula}
If the metric is balanced, then
\begin{equation}\label{eq:S-formula}
 S=\frac{n+1}{2}c\,\Id+\frac14(B-A).
\end{equation}
In particular,
\begin{equation}\label{eq:S-trace}
 \tr S=\frac{n(n+1)}2c.
\end{equation}
\end{lemma}

\begin{proof}
Write
\[
 S_{k\bar\ell}=\ip{Se_k}{\overline{e_\ell}}
 =\sum_iR^b_{i\bar i k\bar\ell}.
\]
Set $j=i$ in \eqref{eq:reconstruction} and sum over $i$.  The constant
curvature part is
\[
 \sum_i\frac c2
 (\delta_{ii}\delta_{k\ell}+\delta_{i\ell}\delta_{ki})
 =\frac{n+1}{2}c\,\delta_{k\ell}.
\]
The first quadratic torsion term is
\[
 -\frac12\sum_{i,r}T^r_{ik}\overline{T^r_{i\ell}}
 =-\frac12A_{k\bar\ell},
\]
where skew-symmetry in the two lower indices is used twice.

The two terms with coefficient $-3/4$ contain, respectively,
\[
 \sum_iT^i_{ir}
 \qquad\text{and}\qquad
 \sum_i\overline{T^i_{ir}}.
\]
They vanish because balancedness gives $\eta_r=\sum_iT^i_{ir}=0$.
The first term with coefficient $1/4$ gives
\[
 \frac14\sum_{i,r}T^\ell_{ir}\overline{T^k_{ir}}
 =\frac14B_{k\bar\ell},
\]
and the second gives
\[
 \frac14\sum_{i,r}T^i_{kr}\overline{T^i_{\ell r}}
 =\frac14A_{k\bar\ell}.
\]
Adding all contributions proves \eqref{eq:S-formula}.  Taking the trace
and using \eqref{eq:A-B-traces} yields \eqref{eq:S-trace}.
\end{proof}

\begin{remark}
The fact that $S$ is a derivation uses only $\nabla^bT=0$.  Formula
\eqref{eq:S-formula} uses the reconstruction formula and balancedness;
balancedness enters exactly in the vanishing of the two contracted
$-3/4$ terms.
\end{remark}

The scalar part in \eqref{eq:S-formula} is the first source of rigidity.
Indeed, a nonzero scalar multiple of the identity is not a derivation of
a nonabelian Lie algebra: the left-hand side of the derivation identity
has one copy of the scalar, whereas the right-hand side has two.  The
remaining sections make this incompatibility precise without assuming
that $A$ and $B$ preserve any preferred splitting.

\section{Excluding a nonzero semisimple quotient}\label{sec:radical}

Let
\[
 \rad=\Rad(\Lie),
 \qquad
 U=\rad^\perp,
 \qquad
 p=\dim_{\C}\rad,
 \qquad
 q=\dim_{\C}U=n-p.
\]
Choose a unitary basis adapted to the orthogonal decomposition
\begin{equation}\label{eq:radical-decomposition}
 V=\rad\oplus U.
\end{equation}
Indices $a,b,d$ range over $\rad$, and indices $\alpha,\beta,\gamma$
range over $U$.  Since $\rad$ is an ideal,
\begin{equation}\label{eq:radical-ideal}
 T(\rad,V)\subseteq\rad.
\end{equation}
Thus the torsion has exactly four possible types:
\begin{equation}\label{eq:four-torsion-types}
 \Lambda^2\rad\longrightarrow\rad,
 \qquad
 U\wedge\rad\longrightarrow\rad,
 \qquad
 \Lambda^2U\longrightarrow\rad,
 \qquad
 \Lambda^2U\longrightarrow U.
\end{equation}
There are no components $\Lambda^2\rad\to U$ or
$U\wedge\rad\to U$.

Using ordered sums, define
\begin{align}
 P&:=\sum_{a,b,d}\abs{T^d_{ab}}^2,
 \label{eq:P-definition}\\
 M&:=\sum_{\alpha,a,b}\abs{T^b_{\alpha a}}^2,
 \label{eq:M-definition}\\
 N&:=\sum_{\alpha,\beta,a}\abs{T^a_{\alpha\beta}}^2,
 \label{eq:N-definition}\\
 L&:=\sum_{\alpha,\beta,\gamma}\abs{T^\gamma_{\alpha\beta}}^2,
 \label{eq:L-definition}\\
 E&:=\sum_r\abs{\sum_aT^a_{ar}}^2.
 \label{eq:E-definition}
\end{align}
The quantities $P,M,N,L$ exhaust all four types in
\eqref{eq:four-torsion-types}.

\subsection{The mixed curvature identity}

\begin{lemma}\label{lem:mixed-identity}
Assume $0<p<n$.  Then
\begin{equation}\label{eq:mixed-identity}
 2pqc=M+3N+2E.
\end{equation}
\end{lemma}

\begin{proof}
Every derivation preserves the radical by
Lemma~\ref{lem:basic-Lie-facts}.  Hence
\[
 D_{\alpha\bar a}e_a\in\rad.
\]
Since $e_\alpha\perp\rad$,
\begin{equation}\label{eq:mixed-curvature-zero}
 R^b_{\alpha\bar a a\bar\alpha}
 =\ip{D_{\alpha\bar a}e_a}{\overline{e_\alpha}}=0.
\end{equation}
Insert
\[
 i=\alpha,
 \qquad
 j=a,
 \qquad
 k=a,
 \qquad
 \ell=\alpha
\]
into \eqref{eq:reconstruction}.  The constant term is $c/2$.  By
skew-symmetry,
\[
 -\frac12\sum_rT^r_{\alpha a}\overline{T^r_{a\alpha}}
 =\frac12\sum_r\abs{T^r_{\alpha a}}^2.
\]
The two $-3/4$ terms are
\[
 -\frac34\sum_r\abs{T^a_{\alpha r}}^2
 \qquad\text{and}\qquad
 -\frac34\sum_r\abs{T^\alpha_{ar}}^2.
\]
The second one vanishes because $T(e_a,e_r)\in\rad$ by
\eqref{eq:radical-ideal}.  The two $1/4$ terms are complex conjugates.
Thus \eqref{eq:mixed-curvature-zero} becomes
\begin{align}
 0={}&\frac c2
 +\frac12\sum_r\abs{T^r_{\alpha a}}^2
 -\frac34\sum_r\abs{T^a_{\alpha r}}^2
 \notag\\
 &\quad
 +\frac12\operatorname{Re}\sum_r
 T^\alpha_{\alpha r}\overline{T^a_{ar}}.
\label{eq:mixed-pointwise}
\end{align}

We now sum over $\alpha$ and $a$.  First,
\begin{equation}\label{eq:mixed-sum-one}
 \sum_{\alpha,a,r}\abs{T^r_{\alpha a}}^2=M,
\end{equation}
because $T(U,\rad)\subseteq\rad$.  Next,
\begin{equation}\label{eq:mixed-sum-two}
 \sum_{\alpha,a,r}\abs{T^a_{\alpha r}}^2=M+N.
\end{equation}
Indeed, $r=b\in\rad$ contributes $M$ after interchanging the dummy
indices $a,b$, while $r=\beta\in U$ contributes $N$.

For the last term, balancedness gives, for each $r$,
\[
 \sum_\alpha T^\alpha_{\alpha r}
 =-\sum_aT^a_{ar}.
\]
Consequently,
\begin{align}\label{eq:mixed-cross-term}
 \operatorname{Re}\sum_{\alpha,a,r}
 T^\alpha_{\alpha r}\overline{T^a_{ar}}
 &=-\sum_r\abs{\sum_aT^a_{ar}}^2
 =-E.
\end{align}
Summing \eqref{eq:mixed-pointwise} and using
\eqref{eq:mixed-sum-one}--\eqref{eq:mixed-cross-term}, we obtain
\[
 0=\frac{pq}{2}c+\frac12M-\frac34(M+N)-\frac12E.
\]
Multiplying by four gives \eqref{eq:mixed-identity}.
\end{proof}

\subsection{Trace on the semisimple quotient}

\begin{lemma}\label{lem:quotient-trace}
Assume $0<p<n$.  Then
\begin{equation}\label{eq:quotient-trace}
 2q(n+1)c=M+N.
\end{equation}
\end{lemma}

\begin{proof}
The quotient $\Lie/\rad$ is semisimple.  Since $S$ is a derivation and
preserves $\rad$, it induces a derivation $\overline S$ of
$\Lie/\rad$.  Every derivation of a complex semisimple Lie algebra is
inner and traceless, so
\begin{equation}\label{eq:quotient-S-zero}
 \tr\overline S=0.
\end{equation}
The orthogonal complement $U$ need not be invariant under $S$.  However,
relative to \eqref{eq:radical-decomposition}, the matrix of $S$ has the
block form
\[
 S=\begin{pmatrix}
 S_{\rad\rad}&S_{\rad U}\\
 0&S_{UU}
 \end{pmatrix},
\]
and $S_{UU}$ represents the induced quotient map.  Hence
\begin{equation}\label{eq:quotient-trace-U}
 \tr\overline S=\tr_U S.
\end{equation}
Using \eqref{eq:S-formula}, we get
\begin{equation}\label{eq:quotient-trace-intermediate}
 0=\frac{q(n+1)}2c
 +\frac14\bigl(\tr_U B-\tr_U A\bigr).
\end{equation}

We retain every torsion type.  From the definition of $A$,
\begin{equation}\label{eq:trace-U-A}
 \tr_U A
 =\sum_{\alpha,s,r}\abs{T^r_{\alpha s}}^2
 =M+N+L.
\end{equation}
For $B$,
\[
 \tr_U B=\sum_{\alpha,r,s}\abs{T^\alpha_{rs}}^2.
\]
An output in $U$ is impossible whenever one input belongs to $\rad$,
because $\rad$ is an ideal.  Therefore only the type
$\Lambda^2U\to U$ remains, and
\begin{equation}\label{eq:trace-U-B}
 \tr_U B=L.
\end{equation}
Thus $L$ is not discarded; it cancels in $\tr_U(B-A)$.  Substituting
\eqref{eq:trace-U-A} and \eqref{eq:trace-U-B} into
\eqref{eq:quotient-trace-intermediate} gives
\eqref{eq:quotient-trace}.
\end{proof}

\begin{proposition}\label{prop:solvable}
If $c\neq 0$, then the torsion Lie algebra $\Lie$ is solvable.
\end{proposition}

\begin{proof}
If $\rad=0$, then $\Lie$ is semisimple.  Since
$S\in\Der(\Lie)$, Lemma~\ref{lem:basic-Lie-facts} gives
$\tr S=0$.  This contradicts \eqref{eq:S-trace} when $c\neq 0$.

Suppose that $0<p<n$.  Combining \eqref{eq:mixed-identity} and
\eqref{eq:quotient-trace} gives
\[
 \frac{p}{n+1}(M+N)=M+3N+2E,
\]
or equivalently,
\begin{equation}\label{eq:positive-combination}
 \left(1-\frac{p}{n+1}\right)M
 +\left(3-\frac{p}{n+1}\right)N
 +2E=0.
\end{equation}
Since $1\leq p\leq n-1$, both displayed coefficients are strictly
positive.  As $M,N,E\geq 0$, equation
\eqref{eq:positive-combination} implies
\[
 M=N=E=0.
\]
Equation~\eqref{eq:quotient-trace} then gives $c=0$, again a
contradiction.  Therefore the only possibility for $c\neq 0$ is
$\rad=\Lie$, that is, $\Lie$ is solvable.
\end{proof}

\begin{remark}
The pure radical quantity $P$ is not assumed to vanish.  It does not
occur in Lemma~\ref{lem:mixed-identity} because one input in the chosen
curvature component always lies in $U$, and it does not occur in
Lemma~\ref{lem:quotient-trace} because the trace is taken over $U$.
Thus arbitrary bracket components internal to the radical are retained.
\end{remark}

\section{The solvable case and adjoint weights}\label{sec:weights}

Assume from now on that $\Lie$ is solvable.  Put
\begin{equation}\label{eq:d-K-definition}
 \der=[\Lie,\Lie],
 \qquad
 K=\der^\perp.
\end{equation}
By Lemma~\ref{lem:A-B-properties},
\begin{equation}\label{eq:K-kernel-B}
 K=\ker B.
\end{equation}
The space $K$ is a Hermitian representative of the abelianization
$\Lie/\der$; it need not be a subalgebra.

\subsection{A derivation lemma for weights}

By Lie's theorem, there is a basis of $\Lie$ in which every
$\ad_x$ is upper triangular.  The diagonal entries are linear
functionals
\[
 \lambda_1(x),\ldots,\lambda_n(x).
\]
The distinct functionals among them are called the weights of the adjoint
representation.

\begin{lemma}
\label{lem:weight-derivation}
Let $\Lambda\subset\Lie^*$ be the finite set of distinct adjoint weights.
Then:
\begin{enumerate}[label=\textup{(\roman*)}]
\item every $\lambda\in\Lambda$ vanishes on
      $\der=[\Lie,\Lie]$;
\item for every $D\in\Der(\Lie)$ and every $\lambda\in\Lambda$,
      \begin{equation}\label{eq:weight-annihilates-D}
       \lambda\circ D=0;
      \end{equation}
\item if every adjoint weight is zero, then $\Lie$ is nilpotent.
\end{enumerate}
\end{lemma}

\begin{proof}
Choose a basis in which all $\ad_x$ are upper triangular.  Since the
commutator of two upper triangular matrices has zero diagonal and
\[
 \ad_{[x,y]}=[\ad_x,\ad_y],
\]
every weight vanishes on $[\Lie,\Lie]$.  This proves (i).

Let $D\in\Der(\Lie)$ and $F_t=\exp(tD)$.  By
Lemma~\ref{lem:basic-Lie-facts}, $F_t$ is an automorphism.  For every
$x\in\Lie$,
\begin{equation}\label{eq:ad-conjugacy}
 \ad_{F_tx}=F_t\,\ad_x\,F_t^{-1}.
\end{equation}
Thus the representations $\ad\circ F_t$ and $\ad$ are equivalent.  The
weight set of $\ad\circ F_t$ is
\[
 \{\lambda\circ F_t:\lambda\in\Lambda\},
\]
whereas equivalent representations have the same finite set of weights.
Therefore
\[
 \{\lambda\circ F_t:\lambda\in\Lambda\}=\Lambda.
\]
For a fixed $\lambda$, the map
$t\mapsto\lambda\circ F_t$ is continuous and takes values in the finite
set $\Lambda$.  It is consequently constant.  At $t=0$ it equals
$\lambda$, so $\lambda\circ F_t=\lambda$ for all $t$.  Differentiating at
$t=0$ gives \eqref{eq:weight-annihilates-D}.

Finally, if every weight is zero, then every $\ad_x$ is strictly upper
triangular and hence nilpotent.  Engel's theorem implies that $\Lie$ is
nilpotent.
\end{proof}

\begin{remark}
The preceding proof uses the finite set of distinct weights, not a chosen
ordering of diagonal entries.  Repeated weights therefore cause no
permutation or differentiability issue.
\end{remark}

\subsection{Curvature on the abelianization}

\begin{lemma}\label{lem:curvature-abelianization}
For every $x\in K=\der^\perp$,
\begin{equation}\label{eq:curvature-abelianization}
 \pr_K\bigl(R^b(x,\overline x)x\bigr)
 =c\norm{x}^2x.
\end{equation}
\end{lemma}

\begin{proof}
The statement is homogeneous in $x$, so assume $\norm{x}=1$.  Choose a
unitary basis with $e_1=x$ and with an initial block spanning $K$.  Since
$K=\ker B$, Lemma~\ref{lem:A-B-properties} gives
\begin{equation}\label{eq:no-output-in-K}
 T^\alpha_{rs}=0
 \qquad\text{whenever }e_\alpha\in K.
\end{equation}
In words, no bracket has a component in $K$.

Let $e_s\in K$.  In \eqref{eq:reconstruction}, set
$i=j=k=1$ and $\ell=s$.  The first torsion term contains
$T(e_1,e_1)$ and vanishes.  Every remaining torsion term contains an
output coefficient in either the $e_1$-direction or the $e_s$-direction,
and hence vanishes by \eqref{eq:no-output-in-K}.  Therefore
\[
 R^b_{1\bar 1 1\bar s}=c\delta_{1s}.
\]
This is exactly
\[
 \pr_K\bigl(R^b(e_1,\overline{e_1})e_1\bigr)=ce_1.
\]
Rescaling proves \eqref{eq:curvature-abelianization}.
\end{proof}

\begin{proposition}\label{prop:nilpotent}
If $c\neq 0$ and $\Lie$ is solvable, then $\Lie$ is nilpotent.
\end{proposition}

\begin{proof}
Assume that $\Lie$ is not nilpotent.  By
Lemma~\ref{lem:weight-derivation}, there is a nonzero adjoint weight
$\lambda$.  Since $\lambda$ vanishes on $\der$, its restriction to the
complementary space $K=\der^\perp$ is nonzero.  Choose $x\in K$ with
$\lambda(x)\neq 0$ and put
\[
 D_x=R^b(x,\overline x).
\]
By Lemma~\ref{lem:curvature-derivations}, $D_x\in\Der(\Lie)$, and
Lemma~\ref{lem:weight-derivation} gives
\begin{equation}\label{eq:weight-Dx-zero}
 \lambda(D_xx)=0.
\end{equation}
On the other hand, $\lambda$ vanishes on
$\der=K^\perp$, so Lemma~\ref{lem:curvature-abelianization} gives
\[
 \lambda(D_xx)
 =\lambda(\pr_KD_xx)
 =c\norm{x}^2\lambda(x),
\]
which is nonzero.  This contradicts \eqref{eq:weight-Dx-zero} and proves
that $\Lie$ is nilpotent.
\end{proof}

\section{Central-direction rigidity}\label{sec:center}

We now prove the decisive pointwise statement.  It does not require
solvability or nilpotency; only the existence of one nonzero central
direction is used.

\begin{lemma}\label{lem:central-direction}
Assume the reconstruction formula \eqref{eq:reconstruction}, and assume
that every curvature operator $D_{i\bar j}$ is a derivation of $\Lie$.
If $c\neq 0$ and $Z(\Lie)\neq 0$, then the bracket of $\Lie$ is zero.
\end{lemma}

\begin{proof}
Choose a unit vector
\[
 0\neq z\in Z(\Lie).
\]
Then
\begin{equation}\label{eq:z-central-input}
 T(z,u)=0\qquad\text{for all }u\in V.
\end{equation}
Centrality eliminates $z$ as an input of the bracket, but it does not
exclude bracket values with a component in $\C z$.  We retain those
components explicitly.  Define the complex skew-symmetric bilinear form
\begin{equation}\label{eq:Omega-z}
 \Omega_z(u,v)=\ip{T(u,v)}{\overline z}
\end{equation}
and the nonnegative Hermitian endomorphism $C_z$ by
\begin{equation}\label{eq:C-z}
 \ip{C_zu}{\overline v}
 =\sum_r\Omega_z(u,e_r)\overline{\Omega_z(v,e_r)}.
\end{equation}
In a unitary basis with $e_1=z$,
\[
 (C_z)_{k\bar\ell}
 =\sum_rT^1_{kr}\overline{T^1_{\ell r}}.
\]
In particular,
\begin{equation}\label{eq:Cz-z-zero}
 C_zz=0.
\end{equation}

Set
\[
 D_z=R^b(z,\overline z).
\]
In \eqref{eq:reconstruction}, take $i=j=1$.  By
\eqref{eq:z-central-input}, every torsion term vanishes except the final
$1/4$ term.  Hence
\begin{equation}\label{eq:Dz-formula}
 D_z=\frac c2\Id+\frac c2P_z+\frac14C_z,
\end{equation}
where $P_z$ is the Hermitian orthogonal projection onto $\C z$.
Formula~\eqref{eq:Dz-formula} shows that $D_z$ is Hermitian self-adjoint
and, by \eqref{eq:Cz-z-zero},
\begin{equation}\label{eq:Dz-z}
 D_zz=cz.
\end{equation}

Since $D_z$ is a derivation,
\begin{equation}\label{eq:Dz-derivation}
 D_zT(u,v)=T(D_zu,v)+T(u,D_zv).
\end{equation}
Pair this identity with $\overline z$.  Using self-adjointness and
\eqref{eq:Dz-z}, the left-hand side becomes
\[
 \ip{D_zT(u,v)}{\overline z}
 =\ip{T(u,v)}{\overline{D_zz}}
 =c\Omega_z(u,v).
\]
After inserting \eqref{eq:Dz-formula} on the right-hand side, the two
scalar terms $c\Id/2$ contribute exactly $c\Omega_z(u,v)$.  The $P_z$
terms vanish by \eqref{eq:z-central-input}.  We obtain
\begin{equation}\label{eq:Omega-C-identity}
 \Omega_z(C_zu,v)+\Omega_z(u,C_zv)=0
\end{equation}
for all $u,v\in V$.

Because $C_z$ is nonnegative Hermitian and $C_zz=0$, choose a unitary
eigenbasis of $C_z$ containing $z$:
\[
 C_ze_i=\lambda_i e_i,
 \qquad
 \lambda_i\geq 0.
\]
Equation~\eqref{eq:Omega-C-identity} gives
\begin{equation}\label{eq:eigen-Omega}
 (\lambda_i+\lambda_j)\Omega_z(e_i,e_j)=0.
\end{equation}
On the other hand, by \eqref{eq:C-z},
\begin{equation}\label{eq:eigenvalue-Omega}
 \lambda_i=\sum_r\abs{\Omega_z(e_i,e_r)}^2.
\end{equation}
If $\Omega_z(e_i,e_j)\neq 0$, then
\eqref{eq:eigenvalue-Omega} and skew-symmetry imply
$\lambda_i>0$ and $\lambda_j>0$, contradicting
\eqref{eq:eigen-Omega}.  Therefore
\begin{equation}\label{eq:Omega-zero}
 \Omega_z=0,
 \qquad
 C_z=0.
\end{equation}
Thus no bracket has a component in the $z$-direction.

By \eqref{eq:Dz-formula} and \eqref{eq:Omega-zero},
\begin{equation}\label{eq:Dz-reduced}
 D_z=\frac c2(\Id+P_z).
\end{equation}
Since $P_zT(u,v)=0$, the left-hand side of
\eqref{eq:Dz-derivation} is
\[
 D_zT(u,v)=\frac c2T(u,v).
\]
Since $z$ is central,
\[
 T(P_zu,v)=T(u,P_zv)=0.
\]
Consequently, the right-hand side of \eqref{eq:Dz-derivation} is
\[
 T(D_zu,v)+T(u,D_zv)=cT(u,v).
\]
It follows that
\[
 \frac c2T(u,v)=cT(u,v).
\]
As $c\neq 0$, we conclude that $T(u,v)=0$ for all $u,v\in V$.
\end{proof}

\begin{remark}
	The first use of the derivation identity eliminates the possible
	\emph{output} of the bracket in the central direction; centrality alone
	eliminates only a bracket \emph{input}.  The second use of the derivation
	identity then eliminates the entire bracket.  This separation is
	necessary, for instance, in the complex Heisenberg algebra, where the
	center is precisely the image of the bracket.
\end{remark}

\begin{proposition}
	\label{prop:zero-torsion-algebra}
	Let $(M^n,J,g)$ be a balanced BTP Hermitian manifold with
	\[
	H^c\equiv0.
	\]
	Fix $p\in M$ and let
	\[
	\mathfrak g_p=(T^{1,0}_pM,T_p)
	\]
	be its torsion Lie algebra. Put
	\[
	\mathfrak r_p=\operatorname{Rad}(\mathfrak g_p),
	\qquad
	U_p=\mathfrak r_p^\perp.
	\]
	Then:
	
	\begin{enumerate}
		\item
		\[
		[U_p,\mathfrak r_p]=0,
		\qquad
		[U_p,U_p]\subset U_p.
		\]
		Consequently,
		\[
		\mathfrak g_p
		=
		\mathfrak r_p\oplus U_p
		\]
		is an orthogonal direct sum of commuting ideals, and $U_p$ is
		semisimple.
		
		\item
		\[
		Z(\mathfrak g_p)\perp
		[\mathfrak g_p,\mathfrak g_p].
		\]
		
		\item If $\mathfrak g_p$ is nilpotent, then it is abelian.
	\end{enumerate}
\end{proposition}

\begin{proof}
	If $\mathfrak r_p=0$ or $\mathfrak r_p=\mathfrak g_p$, the first
	statement is immediate. Suppose
	\[
	0<\dim_{\mathbb C}\mathfrak r_p<n.
	\]
	With the notation of Section~5, the mixed curvature identity at
	$c=0$ gives
	\[
	M+3N+2E=0.
	\]
	Since $M,N,E\ge0$, it follows that
	\[
	M=N=E=0.
	\]
	The equality $M=0$ gives
	\[
	[U_p,\mathfrak r_p]=0,
	\]
	while $N=0$ shows that the $\mathfrak r_p$-component of
	$[U_p,U_p]$ vanishes. Hence $U_p$ is a subalgebra, and the natural
	map
	\[
	U_p\longrightarrow
	\mathfrak g_p/\mathfrak r_p
	\]
	is a Lie algebra isomorphism. Therefore $U_p$ is semisimple and the
	first conclusion follows.
	
	For the second conclusion, let $z\in Z(\mathfrak g_p)$. Repeat the
	first part of the proof of Lemma~\ref{lem:central-direction}. The
	derivation identity gives
	\[
	\Omega_z(C_zu,v)+\Omega_z(u,C_zv)=0.
	\]
	Diagonalizing the nonnegative Hermitian operator $C_z$ shows that
	\[
	\Omega_z=0.
	\]
	Thus
	\[
	\langle [u,v],\overline z\rangle=0
	\]
	for all $u,v\in\mathfrak g_p$, proving
	\[
	Z(\mathfrak g_p)\perp
	[\mathfrak g_p,\mathfrak g_p].
	\]
	
	Finally, suppose that $\mathfrak g_p$ is nilpotent and nonabelian.
	The last nonzero term of its lower central series is a nonzero
	subspace of
	\[
	Z(\mathfrak g_p)\cap
	[\mathfrak g_p,\mathfrak g_p],
	\]
	contradicting the preceding orthogonality. Hence
	$\mathfrak g_p$ is abelian.
\end{proof}

\begin{corollary}
	Under the assumptions of Proposition~\ref{prop:zero-torsion-algebra},
	if the torsion Lie algebra is nilpotent at one point, then $T\equiv0$ and $R^c\equiv0$ on the connected component
	containing $p$.
	In particular, $g$ is flat Kähler on that connected component..
\end{corollary}

\begin{proof}
	Nilpotency implies that the torsion bracket at that point is
	abelian, hence $T=0$ there. Since $\nabla^bT=0$, it follows that
	$T\equiv0$ on the connected component. Thus $g$ is Kähler, and
	$H^c\equiv0$ then implies $R^c\equiv0$ by polarization.
\end{proof}

\section{Proof of the main theorem and consequences}\label{sec:conclusion}

\begin{proof}[Proof of Theorem~\ref{thm:main}]
Fix an arbitrary point $p\in M$ and let
\[
 \Lie_p=(T^{1,0}_pM,T_p)
\]
be its torsion Lie algebra.  Proposition~\ref{prop:solvable} shows that
$c\neq 0$ forces $\Lie_p$ to be solvable.  Proposition~\ref{prop:nilpotent}
then shows that $\Lie_p$ is nilpotent.

If the bracket $T_p$ were nonzero, then $\Lie_p$ would be a nonzero
nilpotent Lie algebra and hence would have nonzero center by
Lemma~\ref{lem:basic-Lie-facts}.  The central-direction lemma would then
force the bracket to vanish, a contradiction.  Thus $T_p=0$.  Since $p$
was arbitrary,
\[
 T\equiv 0.
\]

For completeness, in local holomorphic coordinates one has
\[
 T^k_{ij}=g^{k\bar\ell}
 \bigl(\partial_i g_{j\bar\ell}-\partial_jg_{i\bar\ell}\bigr).
\]
Hence $T=0$ is equivalent to $\partial\omega=0$.  Since $\omega$ is real,
$\bar\partial\omega=0$ as well, and therefore $d\omega=0$.  Thus $g$ is
K\"ahler.
\end{proof}

\begin{proof}[Proof of Corollary~\ref{cor:all-BTP}]
If $g$ is balanced, apply Theorem~\ref{thm:main}.  If $g$ is non-balanced, the existence of
constant Chern holomorphic sectional curvature is ruled out by
\cite[Theorem~1]{ChenZhengBTP}. Hence $g$ must be balanced, and the
result follows.
\end{proof}

\begin{corollary}\label{cor:complete-space-form}
Under the assumptions of Theorem~\ref{thm:main}, suppose in addition
that $(M,g)$ is connected and complete.  Then its universal cover is a
complex space form: a suitably scaled complex projective space if
$c>0$, and a suitably scaled complex hyperbolic space if $c<0$.
\end{corollary}

\begin{proof}
Theorem~\ref{thm:main} makes $g$ K\"ahler.  The conclusion is then the standard classification of complete K\"ahler
manifolds with constant holomorphic sectional curvature \cite{KN}.
\end{proof}

\subsection{The algebraic rigidity theorem}

The geometric proof can be separated completely from differential
geometry once the two BTP inputs have been supplied.

\begin{theorem}\label{thm:algebraic-rigidity}
Let $(V,\ip{\,}{\,})$ be an $n$-dimensional Hermitian vector space, and
let $\mu\in\Lambda^2V^*\otimes V$ satisfy the following conditions:
\begin{enumerate}[label=\textup{(\roman*)}]
\item $\mu$ is a complex Lie bracket;
\item $\tr(\ad_x)=0$ for every $x\in V$;
\item for a real number $c$, choose a unitary basis and define endomorphisms
      $D_{i\bar j}$ by the right-hand side of \eqref{eq:reconstruction},
      replacing $T$ by $\mu$; every $D_{i\bar j}$ is a derivation of
      $(V,\mu)$.  (Equivalently, this holds in every unitary basis.)
\end{enumerate}
If $c\neq 0$, then $\mu=0$.
\end{theorem}

\begin{proof}
The proof is exactly the finite-dimensional argument in
Sections~\ref{sec:contracted}--\ref{sec:center}.  The contracted
curvature derivation and radical calculation make $(V,\mu)$ solvable;
the weight argument makes it nilpotent; and the central-direction lemma
eliminates every nonzero nilpotent bracket.
\end{proof}

\subsection{Check the three-dimensional models}

The three types appearing in the classification of compact balanced
non-K\"ahler BTP threefolds \cite{ZhaoZhengThreefold} provide useful
consistency checks for the three stages of the proof.

In the full-rank case, a unitary frame has
\[
 [e_2,e_3]=ae_1,
 \qquad
 [e_3,e_1]=ae_2,
 \qquad
 [e_1,e_2]=ae_3,
 \qquad a>0.
\]
The torsion algebra is isomorphic to $\mathfrak{sl}_2(\C)$ and is
excluded by the semisimple trace argument.

In the rank-two case, one has
\[
 [e_1,e_3]=ae_1,
 \qquad
 [e_2,e_3]=-ae_2.
\]
This algebra is solvable but not nilpotent and is excluded by the adjoint
weight argument.

In the rank-one case,
\[
 [e_2,e_3]=ae_1.
\]
This is the complex Heisenberg algebra.  Its center is spanned by $e_1$,
and it is excluded by the central-direction lemma.  Thus the pointwise
argument recovers the low-dimensional case division without invoking the
global classification.

\end{document}